\documentclass[letterpaper,12pt,reqno]{amsart}
\usepackage{graphicx,amsmath,amssymb,amsthm}

\usepackage{srcltx}
\usepackage[latin1]{inputenc}
\usepackage[T1]{fontenc}

\usepackage[left=1in,right=1in, top=1in, bottom=1in]{geometry}

\newtheorem{X}{X}[section]

\newtheorem{lemma}[X]{Lemma}
\newtheorem{hypothesis}[X]{Hypothesis}
\newtheorem{proposition}[X]{Proposition}
\newtheorem{theorem}[X]{Theorem}

\theoremstyle{definition}
\newtheorem{remark}[X]{Remark}
\newcommand{\V}{\text{Var}}
\newcommand{\E}{\mathbb E}

\renewcommand{\P}{\text{Prob}}

\title[Non-vanishing of Dirichlet $L$-functions]{On the non-vanishing of Dirichlet $L$-functions at the central point}
\author{Daniel Fiorilli }
\address{Department of Mathematics, University of Michigan, 530 Church Street, Ann Arbor MI 48109 USA}
\email{fiorilli@umich.edu}
\date{\today}

\begin{document}

\maketitle

\begin{abstract}
We investigate the consequences of natural conjectures of Montgomery type on the non-vanishing of Dirichlet $L$-functions at the central point. We first justify these conjectures using probabilistic arguments. We then show using a result of Bombieri, Friedlander and Iwaniec and a result of the author that they imply that almost all Dirichlet $L$-functions do not vanish at the central point. We also deduce a  quantitative upper bound for the proportion of Dirichlet $L$-functions for which $L(\frac 12,\chi)=0$.
\end{abstract}

\section{Introduction and statement of results}
The central values of $L$-functions and their derivatives
are of crucial importance in number theory. Perhaps the most important example are the values $L^{(k)}(E,1)$ for an elliptic curve $E$, which are strongly linked with important invariants of $E$. For $k=1$ this is the Gross-Zagier Formula, and for $k\leq r(E)$ (the rank of $E$), this is the Birch and Swinnerton-Dyer Conjecture.

It is widely believed that the vanishing of $L$-functions at the central point should be explained by arithmetical reasons. The Birch and Swinnerton-Dyer Conjecture is such a reason, and another type of reason is the value of the root number. Indeed, self-dual $L$-functions whose root number is $-1$ must vanish to odd order at the central point. As for Dirichlet $L$-functions, it is believed that we always have $L(\frac 12,\chi)\neq 0$; this was first conjectured by Chowla \cite{Ch} for real primitive characters $\chi$. A good reason to believe this conjecture is that the root number of self-dual Dirichlet $L$-functions, that is $L(s,\chi)$ with $\chi$ real and primitive, can never equal $-1$\footnote{This can be deduced by an exact Gauss sum computation (see Chapters 2 and 9 of \cite{Da})}.

While Chowla's Conjecture is still open, there has been substantial progress towards this question. A famous result of Soundararajan \cite{So} states that the proprotion of Dirichlet $L$-functions $L(s,\chi_{8d})$ with $d$ odd and squarefree which do not vanish at $s=\frac 12$ is at least $\frac 78$; this result was extended by Conrey and Soundararajan \cite{CS} to show that at least $20\%$ of these $L$-functions do not vanish on the whole interval $s\in [0,1]$.

As for general Dirichlet characters $\chi \bmod q$, Balasubramanian and K. Murty \cite{BK}, improving on \cite{B}, have shown that at least $4\%$ of the Dirichlet $L$-functions with $\chi \bmod q$ do not vanish at $s=\frac 12$. This proportion was subsequently improved to $\frac 13$ by Iwaniec and Sarnak\footnote{These authors have also shown \cite{IS2} that in certain families of newforms of either varying weight or level, at least $\frac 12$ of the members satisfy $L(\frac 12,f\otimes \chi_D) >0$, for any fixed $D$. Iwaniec and Sarnak further proved that any improvement of the constant $\frac 12$ would imply a significant bound on Landau-Siegel zeros.} \cite{IS1}, and more recently to $34.11\%$ by Bui \cite{Bu}. Under GRH, Murty \cite{Mu} (see also \cite{Si}) has shown that this proportion is at least $50\%$\footnote{One can interpret this result as an asymptotic for the $1$-level density of low-lying zeros of Dirichlet $L$-functions for test function whose Fourier transform has support contained in $(-2,2)$.}. Sarnak\footnote{Private conversation.} noticed that Montgomery's Conjecture on primes in arithmetic progressions implies the Katz-Sarnak prediction for the $1$-level density for any finite support, and as a consequence almost all Dirichlet $L$-functions do not vanish at the central point. However as we will see below, Montgomery's Conjecture heavily depends on the assumption that $L(\frac 12,\chi)\neq 0$. The goal of the current paper is to formulate an analogue of Montgomery's Conjecture which is independent of real zeros. From this we will deduce that $L(\tfrac 12,\chi)\neq 0$ for almost all $\chi \bmod q$, with $Q< q \leq 2Q$.

We should also mention that corresponding questions for the derivatives $L^{(k)}(\frac 12,\chi)$ have been studied. Bui and Milinovich have shown that asymptotically for $q$ and $k$ tending to infinity, $L^{(k)}(\frac 12,\chi)\neq 0$ for almost all $\chi \bmod q$. A corresponding result for completed Dirichlet $L$-functions $\Lambda(s,\chi)$ had earlier been obtained by Michel and VanderKam \cite{MV}, but with limiting proportion $\frac 23$.

The unconditional results mentioned earlier rely heavily on mollification methods, who have greatly flourished in the past years. The goal of the current paper is to take a different viewpoint to the vanishing of $L(s,\chi)$ at the central point, by inputting probabilistic arguments.

Bombieri, Friedlander and Iwaniec have shown \cite{BFI3} that in the range $Q_0< Q \leq 2 Q_0$, with $Q_0=x^{\frac 12} (\log x)^A$ and $a\neq 0$ a fixed integer,
\begin{equation}
\sum_{\substack{Q<q \leq 2Q \\ (q,a)=1}} \left| \psi(x;q,a) - \frac{\psi(x;\chi_0)}{\phi(q)} \right| \ll  x \left ( \frac{\log\log x}{\log x} \right)^2.
\end{equation}
As a consequence, taking $a=1$ and using the orthogonality relations we obtain the bound
\begin{equation}
\sum_{\substack{Q<q \leq 2Q}} \bigg| \frac 1{\phi(q)}  \sum_{\chi \neq \chi_0} \sum_{\rho_{\chi}} \frac{x^{\rho_{\chi}}}{\rho_{\chi}}  \bigg| \ll  x \left ( \frac{\log\log x}{\log x} \right)^2,
\label{equation corollary of BFI}
\end{equation}
where $\rho_{\chi}$ runs through the nontrivial zeros of $L(s,\chi)$. If $\rho_{\chi} \notin \mathbb R$, then the term $x^{\rho_{\chi}}/\rho_{\chi}$ oscillates; however potential real zeros $\rho_{\chi}$ would result in non-oscillating terms on the left hand side of \eqref{equation corollary of BFI}. It is therefore natural to believe that a better bound holds after removing the real zeros - or at least it is very natural to believe that 
\begin{equation}
\sum_{Q<q \leq 2Q} \bigg| \frac 1{\phi(q)}  \sum_{\chi \neq \chi_0} \sum_{\rho_{\chi}\notin \mathbb R} \frac{x^{\rho_{\chi}}}{\rho_{\chi}}  \bigg| \ll  x \left ( \frac{\log\log x}{\log x} \right)^2.
\label{equation corollary of BFI without real zeros}
\end{equation}
We first remark that this last bound implies the non-vanishing of almost all Dirichlet $L$-functions at the central point.
\begin{proposition}
\label{theorem nonvanishing with BFI}
Fix $A>-2,$ and assume that \eqref{equation corollary of BFI without real zeros} holds in the range $x^{\frac 12} (\log x)^{A}< Q \leq 2 x^{\frac 12} (\log x)^A$. Then almost all Dirichlet $L$-functions do not vanish at the central point. More precisely, for $Q$ large enough we have
\begin{equation}
 \frac 1{Q^2} \sum_{q \leq Q} \sum_{\chi \bmod q} z(\chi) \ll \frac {(\log\log Q)^2}{(\log Q)^{2+A} },
 \label{equation nonvanishing}
\end{equation} 
where $z(\chi)$ is the number of real zeros of $L(s,\chi)$ in the critical strip, counted with multiplicity. 
\end{proposition}
\begin{remark}
Montgomery's probabilistic argument (see below) supports \eqref{equation corollary of BFI without real zeros} (and predicts a stronger bound). As for \eqref{equation corollary of BFI} (which is known unconditionally), one would need to add the assumption that $L(\frac 12,\chi)\neq 0$ for Montgomery's argument to support this bound.  
\end{remark}

We now investigate the implications of a more powerful conjecture than \eqref{equation corollary of BFI without real zeros} on the non-vanishing of Dirichlet $L$-functions at the central point. Montgomery's Conjecture, which is motivated by a probabilistic argument, states that in a certain range of $q$ and $x$ with $(a,q)=1$,
$$ \psi(x;q,a) - \frac{\psi(x;\chi_0)}{\phi(q)}  = \frac 1{\phi(q)} \sum_{\chi\neq \chi_0} \overline{\chi}(a) \psi(x;\chi)  \ll_{\epsilon} \frac{x^{\frac 12+\epsilon}}{q^{\frac 12}}. $$
This conjecture is based on the fact that under GRH we have
\begin{equation}
 \frac {x^{-\frac 12}}{\phi(q)} \sum_{\chi\neq \chi_0} \overline{\chi}(a) \psi(x;\chi) = -\frac 1{\phi(q)}\sum_{\chi\neq \chi_0} \overline{\chi}(a) \sum_{\rho_{\chi}} \frac{x^{ i\gamma_{\chi}}}{\rho_{\chi}}+O(x^{-\frac 12}(\log qx)^2),
 \label{equation explicit formula}
\end{equation}  
and one can show (see Appendix \ref{appendix}) that if the $\gamma_{\chi}$ are distinct and nonzero, then the first term on the right hand side of \eqref{equation explicit formula} has a limiting logarithmic distribution with zero mean and variance $\asymp \phi(q)^{-1}\log q$. Hence we believe that this term should not exceed $q^{-\frac 12+\epsilon}$. If we remove the assumption that the $\gamma_{\chi}$ are nonzero, then we need to reformulate Montgomery's Conjecture. Indeed if the proportion of $\chi \bmod q$ such that $L(\frac 12 ,\chi)=0 $ is not exactly zero, then Montgomery's Conjecture is false\footnote{\label{footnote sarnak}The contrapositive of this statement follows from Theorem 2.13 of \cite{FiMi}. Indeed, taking the test function $\eta_{\kappa}(y):=(\sin(\kappa\pi y )/\kappa\pi y)^2$, whose Fourier transform is supported in the interval $[-\kappa,\kappa]$, the $1$-level density is asymptotically $\widehat{\eta_{\kappa}}(0)=1/\kappa$, and taking arbitrarily large values of $\kappa$ gives the desired conclusion. 
}. We now reformulate this conjecture, depending on a parameter $0<\eta<1$.
\begin{hypothesis}[Modified Montgomery Conjecture]
\label{hypothesis reformulated montgomery}
Fix $\epsilon>0$. In the range $q \leq x^{\eta}$, we have for $(a,q)=1$ that
\begin{equation}
\frac 1{\phi(q)}\sum_{\chi\neq \chi_0} \overline{\chi}(a) \sum_{\rho_{\chi} \notin \mathbb R} \frac{x^{ \rho_{\chi}}}{\rho_{\chi}} \ll_{\epsilon} \frac{x^{\frac 12+\epsilon}}{q^{\frac12}}.
\label{equation strong hypothesis}
\end{equation} 
\end{hypothesis}

We will show that this hypothesis implies a strong non-vanishing result on Dirichlet $L$-functions at the central point.
\begin{theorem}
\label{theorem strong nonvanishing}
Fix $\epsilon>0$. Assume GRH, and assume that for some $\frac 12< \eta<1$, Hypothesis \ref{hypothesis reformulated montgomery} holds\footnote{It is actually sufficient to assume that \eqref{equation strong hypothesis} holds on average over $q \leq Q$, with $Q\leq x^{\eta}$ and $a=1$.}. Then we have that
\begin{equation}
\frac 1{Q^2} \sum_{q\leq Q} \sum_{\chi \bmod q} z(\chi) \ll_{\epsilon} \frac 1{Q^{\frac 12-\epsilon}},
\label{equation nonvanishing strong}
\end{equation}  
where $z(\chi)$ is the order of vanishing of $L(s,\chi)$ at $s=\frac 12$. 
\end{theorem}

\begin{remark}In contrast with Montgomery's Conjecture, Hypothesis \ref{hypothesis reformulated montgomery} does not imply GRH, but rather implies that the nonreal zeros of $L(s,\chi)$ lie on the line $\Re(s)=\frac 12$. This last statement was used as a hypothesis in the work of Sarnak and Zaharescu \cite{SZ}, who showed that it implies an effective bound on the class number of imaginary quadratic fields.
\end{remark}

\begin{remark}
As mentioned earlier in Footnote \ref{footnote sarnak}, Montgomery's Conjecture implies the Katz-Sarnak prediction for the $1$-level density in the family of Dirichlet $L$-functions modulo $q$, and thus it follows that almost all members of this family do not vanish at the central point. However, Montgomery's Conjecture has the assumption that $L(\frac 12,\chi)\neq 0$ built in, and moreover the Katz-Sarnak density conjecture does not allow one to obtain an explicit error term as in \eqref{equation nonvanishing} and \eqref{equation nonvanishing strong}. 
Finally, the range $x^{\frac 12-o(1)} < Q< x^{\frac 12+o(1)} $ in which we are working in Proposition \ref{theorem nonvanishing with BFI} corresponds in the Katz-Sarnak problem to test functions whose Fourier transform is supported in $(-2-\epsilon,-2+\epsilon) \cup (2-\epsilon,2+\epsilon)$, and thus does not allow one to tackle the rest of the support, which is needed to obtain the non-vanishing of almost-all Dirichlet $L$-functions at the central point.
\end{remark}

Proposition \ref{theorem nonvanishing with BFI} follows from a fairly straightforward argument. As for Theorem \ref{theorem strong nonvanishing}, the proof is more involved and relies on the properties of the sum
$$S(Q;x):=-\sum_{ Q<q\leq 2Q} \left( \psi(x;q,1) - \frac{\psi(x,\chi_0)}{\phi(q)} \right), $$
which we will study using two different techniques. We now record one of the resulting estimates which we believe is of independent interest.

\begin{proposition}
\label{proposition strong estimate for S(Q;x)}
Fix $\epsilon>0$, assume GRH and assume that Hypothesis \ref{hypothesis reformulated montgomery} holds for some $\frac 12 < \eta <1$. Then in the range $x^{\frac 12}<Q \leq x$ we have that
$$ S(Q;x) = \frac Q2 \log(x/Q) +C_3 Q +O_{\epsilon} \left(\frac{x^{1+\epsilon}} {Q^{\frac 12}}  +Q^{\frac 32-\epsilon} x^{-\frac 12+\epsilon}\right),  $$
where $$ C_3:= \frac 12\left( \log 2\pi + \gamma +\sum_p \frac{\log p}{p-1} +1 \right)-\log 2. $$
(Note that this gives an asymptotic for $S(Q,x)$ in the range $x^{\frac 23+o(1)} < Q = o(x)$, and that the error term is independent of $\eta$.)
\end{proposition}

\section{An application of the Bombieri-Friedlander-Iwaniec Theorem}
In this section we prove Proposition \ref{theorem nonvanishing with BFI}.
\begin{proof}[Proof of Proposition \ref{theorem nonvanishing with BFI}]
Applying the triangle inequality twice gives that in the range \linebreak $x^{\frac 12} (\log x)^{A}< Q \leq 2 x^{\frac 12} (\log x)^A$,
\begin{align*}
\bigg|\sum_{Q<q\leq 2Q} \frac 1{\phi(q)} \sum_{\chi \neq \chi_0} \sum_{\rho_{\chi}\in \mathbb R} \frac{x^{\rho_{\chi}}}{\rho_{\chi}} \bigg| &\leq \sum_{Q<q\leq 2Q} \bigg|\frac 1{\phi(q)} \sum_{\chi \neq \chi_0} \sum_{\rho_{\chi} } \frac{x^{\rho_{\chi}}}{\rho_{\chi}} \bigg| + \sum_{Q<q\leq 2Q} \bigg|\frac 1{\phi(q)} \sum_{\chi \neq \chi_0} \sum_{\rho_{\chi} \notin \mathbb R} \frac{x^{\rho_{\chi}}}{\rho_{\chi}} \bigg|\\
& \ll x  \left(\frac{\log\log x}{\log x}\right)^2,
\end{align*}  
by \eqref{equation corollary of BFI} and \eqref{equation corollary of BFI without real zeros}. We therefore have that
\begin{equation}
x \left ( \frac{\log\log x}{\log x} \right)^2\gg \sum_{Q<q\leq 2Q} \frac 1{\phi(q)} \sum_{\chi \neq \chi_0} \sum_{\rho_{\chi}\in \mathbb R} \frac{x^{\rho_{\chi}}}{\rho_{\chi}} \geq \sum_{Q<q\leq 2Q} \frac 1{q} \sum_{\chi \neq \chi_0} \sum_{\rho_{\chi}\in \mathbb R} \frac{x^{\rho_{\chi}}}{\rho_{\chi}}.
\label{equation proof 1}
\end{equation} 
We now note that if $\rho \in \mathbb R$ is a zero of $L(s,\chi)$, then $1-\rho$ is also a zero of $L(s,\overline{\chi})$ with the same multiplicity $m_{\rho}$, hence this pair of zeros give a contribution of 
$$ m_{\rho}\frac{x^{\rho}}{\rho} +m_{\rho} \frac{x^{1-\rho}}{1-\rho} \geq m_{\rho} x^{\frac 12}.$$
Therefore grouping characters by conjugate pairs in \eqref{equation proof 1}, we obtain that the last term on the right is 
$$\geq \sum_{Q<q\leq 2Q} \frac 1{q} \sum_{\chi \neq \chi_0} \frac 12 \sum_{\rho_{\chi}\in \mathbb R} x^{\frac 12} \geq \frac {x^{\frac 12}}{4Q} \sum_{Q<q\leq 2Q} \sum_{\chi \neq \chi_0}z(\chi).$$

We conclude that
$$\frac 1{Q^2} \sum_{Q<q\leq 2Q} \sum_{\chi \bmod q}z(\chi) \ll \frac{x^{\frac 12}} Q \left ( \frac{\log\log x}{\log x} \right)^2.$$
A standard argument using dyadic intervals gives the claimed bound.
\end{proof}

\section{Applications of Montgomery's Conjecture}
In this section we study the quantity
\begin{equation}
\label{equation definition S(Q;x)}
S(Q;x)=-\sum_{ Q<q\leq 2Q} \left( \psi(x;q,1) - \frac{\psi(x,\chi_0)}{\phi(q)} \right), 
\end{equation}
using two different techniques. The proof of Theorem \ref{theorem strong nonvanishing} will follow by comparing these two estimates.

We first give a conditional bound on $S(Q;x)$ using techniques of \cite{fiorilli}, which ultimately relies on Hooley's variant of the divisor switching method \cite{H}.
\begin{lemma}
\label{lemma weak divisor switch}
Fix $\epsilon>0$ and assume GRH. In the range $x^{\frac 12} \leq  Q \leq x$, we have the estimate
\begin{equation} S(Q;x)= \frac Q2 \log (x/Q) + C_3 Q +O_{\epsilon} \left(\frac{x^{\frac 32} (\log x)^2} Q  +Q^{\frac 32-\epsilon} x^{-\frac 12+\epsilon}\right),
\end{equation}
where $S(Q;x)$ is defined in \eqref{equation definition S(Q;x)} and
$$ C_3:= \frac 12\left( \log 2\pi + \gamma +\sum_p \frac{\log p}{p(p-1)} +1 \right)-\log 2. $$
(Note that this is gives an asymptotic for $S(Q;x)$ in the range $x^{\frac 34} \log x = o(Q)$, $Q\leq x$.)
\end{lemma}
\begin{proof}
We evaluate $S(Q,x)$ by following the argument in the proof of Proposition 6.1 of \cite{fiorilli} (see also \cite{Fo, FrGr, FGHM,H}). We first write
$$ S(Q,x)= \sum_{2Q<q\leq x} \psi(x;q,1) -\sum_{Q<q\leq x} \psi(x;q,1) + \psi(x;\chi_0) \sum_{Q< q\leq 2Q} \frac 1{\phi(q)}=I-II+III. $$
Lemma 5.2 of \cite{fiorilli} combined with the Riemann Hypothesis implies that
 $$III = C_1 x \log 2  +O\left( x\frac{\log Q}Q +x^{\frac 12}(\log x)^2\right),$$
 where $C_1:= \zeta(2)\zeta(3)/\zeta(6)$. We treat $I$ and $II$ as follows (see Lemma 5.1 of \cite{fiorilli}):
\begin{align}
II&=\sum_{Q<q\leq x} \sum_{\substack{ Q< n\leq x \\ q \mid n-1}} \Lambda(n) = \sum_{\substack{1\leq r <(x-1)/Q}} \sum_{\substack{rQ+1 < n \leq x \\ r \mid n-1 }} \Lambda(n) \notag\\
&= \sum_{\substack{1\leq r <(x-1)/Q}} (\psi(x;r,1)-\psi(rQ+1;r,1)) \label{equation to iterate 1} \\
&= \sum_{\substack{1\leq r <(x-1)/Q}} \frac{x-rQ-1}{\phi(r)} + O\left(  \frac{x^{\frac 32} (\log x)^2} Q\right) \label{equation to iterate 2} \\
&=x \left(C_1 \log(x/Q) +C_2 + \frac{\log (x/Q)}{2x/Q} +\frac{C_0 Q}x + O_{\epsilon}\left(\left(\frac Qx \right)^{\frac 32-\epsilon} +\frac{x^{\frac 12} (\log x)^2} Q \right)\right)\notag
\end{align}
by GRH and Lemma 5.9 of \cite{fiorilli}. Here,
$$ C_2:= C_1 \left( \gamma-1 -\sum_p \frac{\log p}{p^2-p+1} \right), \hspace{1cm} C_0:= \frac 12\left( \log 2\pi + \gamma +\sum_p \frac{\log p}{p(p-1)} +1 \right). $$
Note that at this point we cannot apply Hypothesis \ref{hypothesis reformulated montgomery} in going from \eqref{equation to iterate 1} to \eqref{equation to iterate 2}, since we have no information on the real zeros of $L(s,\chi)$. Later we will reiterate this proof and apply our non-vanishing results at this step to get a better error term. 

We conclude the proof by collecting our estimates for $I$, $II$ and $III$:
\begin{align*}
S(Q,x) &=  \frac Q2 \log(x/Q) +Q( C_0-\log 2) +O_{\epsilon}\left( x^{\frac 12}(\log x)^2+ Q^{\frac 32-\epsilon} x^{-\frac 12+\epsilon} +\frac{x^{\frac 32} (\log x)^2} Q\right).\\
\end{align*}  
\end{proof}

We now combine Lemma \ref{lemma weak divisor switch} with Hypothesis \ref{hypothesis reformulated montgomery} to obtain a first non-vanishing result.

\begin{lemma}
\label{lemma weaker result}
Fix $\epsilon>0$. Assume GRH, and assume that for some $\frac 12< \eta<1$, Hypothesis \ref{hypothesis reformulated montgomery} holds\footnote{It is actually sufficient to assume that for $Q\asymp x^{\eta}$, $ \sum_{Q<q \leq 2Q}\frac 1{\phi(q)}\sum_{\chi\neq \chi_0} \sum_{\rho_{\chi} \notin \mathbb R} \frac{x^{ \rho_{\chi}}}{\rho_{\chi}} \ll_{\epsilon} Q^{\frac 12}x^{\frac 12+\epsilon}.$}. Then we have that
\begin{equation}
\frac 1{Q^2} \sum_{Q<q\leq 2Q} \sum_{\chi \bmod q} z(\chi) \ll_{\epsilon} \frac 1{Q^{\min(\frac 12,2-\frac 1{\eta})-\epsilon}}. 
\label{equation nonvanishing weak}
\end{equation}  
where $z(\chi)$ is the number of real zeros of $L(s,\chi)$ in the critical strip, counted with multiplicity. (Note that if $\frac 23 \leq \eta < 1$, then the right hand side of \eqref{equation nonvanishing weak} equals $Q^{-\frac 12+\epsilon}$).
\end{lemma}
\begin{proof}
We study the quantity
$$S(Q;x)=-\sum_{ Q<q\leq 2Q} \left( \psi(x;q,1) - \frac{\psi(x,\chi_0)}{\phi(q)} \right), $$
in the range $x^{\eta}/3 \leq Q\leq x^{\eta}/2$.

On one hand, we apply the explicit formula and GRH:
\begin{align}
S(Q;x)  &=  \sum_{ Q<q\leq 2Q} \frac 1{\phi(q)}\sum_{\chi\neq \chi_0} \sum_{\rho_{\chi} } \frac{x^{ \rho_{\chi}}}{\rho_{\chi}} +O(Q (\log Qx)^2) \notag \\&= 2x^{\frac 12}\sum_{ Q<q\leq 2Q} \frac 1{\phi(q)} \sum_{\chi \neq \chi_0} z(\chi) +\sum_{ Q<q\leq 2Q} \frac 1{\phi(q)}\sum_{\chi\neq \chi_0} \sum_{\rho_{\chi} \notin \mathbb R } \frac{x^{\rho_{\chi}}}{\rho_{\chi}} +O(Q (\log Qx)^2) \notag \\
& \geq \frac{x^{\frac 12}}{Q} \sum_{Q<q\leq 2Q} \sum_{\chi \neq \chi_0} z(\chi) +O_{\epsilon}(x^{\frac 12+\frac{\epsilon}2} Q^{\frac 12}),
\label{equation first estimate for S}
\end{align}  
by Hypothesis \ref{hypothesis reformulated montgomery}.

On the other hand, we compare this with the estimate for $S(Q;x)$ in Lemma \ref{lemma weak divisor switch}, yielding
$$\frac 1{Q^2}\sum_{Q<q\leq 2Q} \sum_{\chi \neq \chi_0} z(\chi) \ll_{\epsilon} x^{\frac{\epsilon}2} Q^{-\frac 12}+  \frac{\log x}{x^{\frac 12}} + \frac{x (\log x)^2} {Q^2} \ll_{\epsilon} Q^{\epsilon-\frac 12} + Q^{\frac 1{\eta} -2+\epsilon}. $$
\end{proof}

We now refine Lemma \ref{lemma weaker result}, by re-inserting Hypothesis \ref{hypothesis reformulated montgomery} in its proof. We will iterate this process several times, until we reach the error term appearing in Theorem \ref{theorem strong nonvanishing}.

\begin{lemma}
Fix $\epsilon>0$, assume GRH and assume that Hypothesis \ref{hypothesis reformulated montgomery} holds\footnote{Again it is sufficient to assume that \eqref{equation strong hypothesis} holds on average over $q \leq Q$, with $Q\leq x^{\eta}$ and $a=1$.} for some $\frac 12 < \eta < 1$. Assume further that for $\kappa(\eta)$ a function of $\eta$ satisfying $ 0<\kappa(\eta)<\frac 12$, we have 
\begin{equation}
\frac 1{Q^2}\sum_{Q< q\leq 2Q} \sum_{\chi \neq \chi_0} z(\chi)  \ll_{\epsilon} \frac 1{Q^{\frac 12-\epsilon}}+
\frac 1{Q^{\kappa(\eta) -\epsilon}}.
\label{estimate to improve}
\end{equation} 
Then it follows that
\begin{equation}
\frac 1{Q^2}\sum_{Q<q\leq 2Q} \sum_{\chi \neq \chi_0} z(\chi)  \ll_{\epsilon} \frac 1{Q^{\frac 12-\epsilon}}+\frac 1{Q^{2-\frac 1{\eta}  -\kappa(\eta)(1-\frac 1{\eta}) -\epsilon}}.
\label{estimate improved}
\end{equation} 
\label{lemma iterative process}
\end{lemma}

\begin{proof}
We set $x^{\eta}/3 \leq Q \leq x^{\eta}/2$ and follow the proofs of Lemmas \ref{lemma weak divisor switch} and \ref{lemma weaker result}, applying \eqref{estimate to improve} in going from \eqref{equation to iterate 1} to \eqref{equation to iterate 2}. Note that \eqref{estimate to improve}, GRH and Hypothesis \ref{hypothesis reformulated montgomery} imply that
\begin{align*}
\sum_{\substack{1\leq r <(x-1)/Q}} & \left(\psi(x;r,1) -\psi(rQ+1;r,1) - \frac {\psi(x,\chi_0)-\psi(rQ+1,\chi_0)}{\phi(r)}\right)  \\
&= \sum_{\substack{1\leq r <(x-1)/Q}}  \frac 1{\phi(r)}\sum_{\chi \neq \chi_0} (\psi(x,\chi)-\psi(rQ+1,\chi)) \\
&= \sum_{\substack{1\leq r <(x-1)/Q}}  \frac 1{\phi(r)} \sum_{\chi \neq \chi_0}\sum_{\gamma_{\chi}\neq 0} \frac{x^{\frac 12+i\gamma_{\chi}} - (rQ+1)^{\frac 12+i\gamma_{\chi}}}{\frac 12+i\gamma_{\chi}}
\\ & \hspace{1cm}+  2\sum_{\substack{1\leq r <(x-1)/Q}}  \frac 1{\phi(r)} \sum_{\chi\neq \chi_0}z(\chi) (x^{\frac 12} - (rQ+1)^{\frac 12}) \\
&\ll_{\epsilon} \frac{x^{1+\epsilon}} {Q^{\frac 12}} + x^{\frac 32-\kappa(\eta)+\epsilon}Q^{-1+\kappa(\eta)},
\end{align*} 
since for $r <(x-1)/Q$ we always have $r\leq (rQ+1)^{\eta}$, thanks to the fact that $\frac 12 < \eta < 1$. Also in applying \eqref{estimate to improve} we used a dyadic decomposition of the sum over $r$. Following the subsequent steps of the proofs of Lemmas \ref{lemma weak divisor switch} and \ref{lemma weaker result} we obtain that since $Q\geq x^{\frac 12}$,
$$\frac 1{Q^2}\sum_{Q<q\leq 2Q} \sum_{\chi \neq \chi_0} z(\chi) \ll_{\epsilon} \frac{x^{\frac{\epsilon}2}}{Q^{\frac 12}}+ x^{1-\kappa(\eta)+\frac{\epsilon}2}Q^{-2+\kappa(\eta)}+x^{\frac 12+\epsilon}Q^{-\frac 32}\ll_{\epsilon} \frac 1{Q^{\min(\frac 12,2-\frac 1{\eta} -\kappa(\eta)(1-\frac 1{\eta}))-\epsilon}}. $$

\end{proof}

We now show that starting from Lemma \ref{lemma weaker result} with a fixed $\frac 12<\eta< 1$ and applying Lemma \ref{lemma iterative process} iteratively, we eventually obtain the error term $Q^{-\frac 12+\epsilon}$.

\begin{lemma}
\label{lemma iterative process stops}
Fix $\frac 12< \eta <1$, and define $f(t):=2-\frac 1{\eta} - t (1-\frac 1{\eta} )$.
Then for $n$ large enough (depending on $\eta$), we have that 
$$ f^{(n)}\left(2-\frac 1\eta\right) >\frac 12, $$
where $f^{(n)}$ is the $n$-th iterate of $f$.
\end{lemma}

\begin{proof}
One easily shows the following formula:
$$f^{(n)}\left(2-\frac 1\eta\right)  =\left( 2-\frac 1{\eta} \right) \sum_{k=0}^n \left( \frac 1{\eta}-1 \right)^k = 1- \left( \frac 1{\eta}-1\right)^{n+1}. $$
It follows that for any fixed $\frac 12 < \eta < \infty,$
$$ \lim_{n\rightarrow \infty}  f^{(n)}\left(2-\frac 1\eta\right)  = 1.$$

\end{proof}

We are now ready to prove Theorem \ref{theorem strong nonvanishing}.

\begin{proof}[Proof of Theorem \ref{theorem strong nonvanishing}]
Fix $\frac 12<\eta<1$. By Lemma \ref{lemma weaker result}, we have that
\begin{equation}
\frac 1{Q^2} \sum_{Q<q\leq 2Q} \sum_{\chi \bmod q} z(\chi) \ll_{\epsilon} \frac 1{Q^{\frac 12-\epsilon }}+ \frac 1{Q^{2-\frac 1{\eta}-\epsilon}}. 
\end{equation}  
We apply Lemma \ref{lemma iterative process} iteratively to this estimate; Lemma \ref{lemma iterative process stops} implies that after a finite number of steps we will obtain the bound 
\begin{equation}
\frac 1{Q^2} \sum_{Q<q\leq 2Q} \sum_{\chi \bmod q} z(\chi) \ll_{\epsilon} \frac 1{Q^{\frac 12-\epsilon }}.
\end{equation}  
The desired estimate follows from a decomposition into dyadic intervals.

\end{proof}

\begin{proof}[Proof of Proposition \ref{proposition strong estimate for S(Q;x)}]
We follow the proof of Lemma \ref{lemma weak divisor switch}. We apply Theorem \ref{theorem strong nonvanishing} in going from \eqref{equation to iterate 1} to \eqref{equation to iterate 2}; as seen in the proof of Lemma \ref{lemma iterative process}, this will yield that 
$$ S(Q;x) = \frac Q2 \log(x/Q) +C_3Q + O_{\epsilon}\left( x^{\frac 12} (\log x)^2  + Q^{\frac 32-\epsilon} x^{-\frac 12+\epsilon} + \frac{x^{1+\epsilon}}{Q^{\frac 12}}\right). $$
The proof follows since $x^{\frac 12}<Q \leq x$.
\end{proof}

\appendix
\section{The distribution of the error term in the prime number theorem in arithmetic progressions}
\label{appendix}
In this appendix we study the limiting logarithmic distribution of the term on the left hand side of \eqref{equation strong hypothesis}, and justify Hypothesis \ref{hypothesis reformulated montgomery}. Let us first study the remainder term in the prime number theorem for arithmetic progressions: 
$$T(x;q,a):= -x^{-\frac 12}\left( \psi(x;q,a)-\frac{\psi(x,\chi_0)}{\phi(q)}\right) =\frac {x^{-\frac 12}}{\phi(q)}\sum_{\chi\neq \chi_0} \overline{\chi}(a) \sum_{\rho_{\chi}} \frac{x^{ \rho_{\chi}}}{\rho_{\chi}}+o(1).$$
Assuming GRH, one can show that $T(x;q,a)$ has a limiting logarithmic distribution $\mu_{q;a}$, a probability measure whose associated random variable will be denoted by $X_{q;a}$.

\begin{proposition}
\label{proposition limiting distribution}
Assume GRH. Then $T(e^y;q,a)$ has a limiting probability distribution $\mu_{q;a}$ as $y \rightarrow \infty$, whose mean is given by 
$$ \E[X_{q;a}]=\int_{\mathbb R} t d\mu_{q;a}(t) = \frac 2{\phi(q)}\sum_{\chi \neq \chi_0} \overline{\chi}(a) z(\chi),$$
where $z(\chi)$ is the order of vanishing of $L(s,\chi)$ at $s=\frac 12$. The variance of $X_{q;a}$ is given by 
$$\V[X_{q;a}]= \int_{\mathbb R} (t-\E[X_{q;a}])^2 d\mu_{q;a}(t) = \frac 1{\phi(q)^2}\sum_{\chi \neq \chi_0} |\chi(a)|^2\sum_{\gamma_{\chi}\neq 0} \frac{m_{\gamma_{\chi}}^2}{\frac 14+ \gamma_{\chi}^2}, $$
where $m_{\gamma_{\chi}}$ denotes the multiplicity of $\gamma_{\chi}$ in the multiset 
$S(q):= \{ \gamma_{\chi} : L(\frac 12+i\gamma_{\chi},\chi)=0,\chi \bmod q \}$.
\end{proposition}
\begin{proof}
The existence of the limiting distribution follows from \cite{NgSh}. The computation of the first two moments is almost identical to that in Lemmas 2.4 and 2.5 of \cite{Fi}.

\end{proof}

We now study the left hand side of \eqref{equation strong hypothesis} by defining
$$T^*(x;q,a):= \frac {x^{-\frac 12}}{\phi(q)}\sum_{\chi\neq \chi_0} \overline{\chi}(a) \sum_{\rho_{\chi}\notin \mathbb R} \frac{x^{ \rho_{\chi}}}{\rho_{\chi}}.$$
Similarly as in Proposition \ref{proposition limiting distribution}, one shows under GRH that $T^*(x;q,a)$ has a limiting logarithmic distribution whose mean is exactly zero and whose variance is given by 
$$ V^*(q;a):=\frac 1{\phi(q)^2}\sum_{\chi \neq \chi_0} |\chi(a)|^2\sum_{\gamma_{\chi}\neq 0} \frac{m_{\gamma_{\chi}}^2}{\frac 14+ \gamma_{\chi}^2}. $$
Assuming that the $m_{\gamma_{\chi}}$ are uniformly bounded, we deduce using the Riemann-von Mangoldt formula that $V^*(q;a) \asymp \phi(q)^{-1}\log q$. Hence, if $\Psi(q)$ is any function tending to infinity, then Chebyshev's Inequality gives
$$\P[|X_{q;a}| \geq \Psi(q) \phi(q)^{-\frac 12} (\log q)^{\frac 12}] \ll \frac 1{\Psi(q)^2}, $$
that is $X_{q;a}$ is normally bounded above by $\phi(q)^{-\frac 12} (\log q)^{\frac 12}$. We need however to be careful in making conjectures about the size of $T^*(x;q,a)$, since even though very rare, 'Littlewood phenomena' do happen. For this reason we add the $x^{\epsilon}$ factor, which gives \eqref{equation strong hypothesis}. We should  also be careful with the range $q\leq x^{\eta}$ in \eqref{equation strong hypothesis}, since by the work of Friedlander and Granville \cite{FG1,FG2}, the primes up to $x$ are not equidistributed in arithmetic progressions modulo $q$ when $q \asymp x/(\log x)^B$ .
\section*{Acknowledgements}
This work was supported by an NSERC Postdoctoral Fellowship, and was accomplished at the University of Michigan.

\end{document}